\documentclass{dcds-bEA}
\usepackage{amsmath}
\usepackage{amssymb}
\usepackage{paralist}
\usepackage[misc]{ifsym}
\usepackage{epsfig}
\usepackage{epstopdf}
\usepackage{graphicx}
\usepackage{booktabs}
\usepackage[colorlinks=true]{hyperref}
\hypersetup{urlcolor=blue, citecolor=red}
\allowdisplaybreaks

\usepackage{amsfonts}
\usepackage{mathtools}
\usepackage{bm}
\usepackage{subcaption}
\usepackage{placeins}
\usepackage{algorithm}
\usepackage{algpseudocode}
\usepackage{float}
\usepackage{tabularx}
\usepackage{makecell}
\usepackage{array}

\newcolumntype{C}{>{\centering\arraybackslash}X}
\newcolumntype{L}{>{\raggedright\arraybackslash}X}
\newcolumntype{R}{>{\raggedleft\arraybackslash}X}

\def\currentvolume{X}
\def\currentissue{X}
\def\currentyear{202X}
\def\currentmonth{XX}
\def\ppages{1--XX}
\def\DOI{10.3934/dcdsb.2026072}

\newtheorem{proposition}{Proposition}

\title[Optimal chemo--immunotherapy scheduling]
{Optimal chemo--immunotherapy scheduling: a hybrid QPSO--SQP approach with Michaelis--Menten pharmacodynamics}

\author[B. Kidane, M. Motayed and S. Wang]{}


\thanks{$^*$Corresponding author: Shuo Wang}

\begin{document}

\maketitle

\centerline{\scshape
		Bereket Sitotaw Kidane$^{{\href{mailto:bereket.kidane@uta.edu}{\textrm{\Letter}}}}$, Md Samiul Haque Motayed$^{{\href{mailto:mdsamiulhaque.motayed@uta.edu}{\textrm{\Letter}}}}$}
		\centerline{\scshape and Shuo Wang$^{{\href{mailto:shuolinda.wang@uta.edu}{\textrm{\Letter}}}*}$}

\medskip

{\footnotesize
 \centerline{Department of Mechanical and Aerospace Engineering}
 \centerline{The University of Texas at Arlington, Arlington, TX 76019, USA}
}

\bigskip

\centerline{(Communicated by Urszula Ledzewicz)}

\begin{abstract}
Optimal scheduling of combined chemo--immunotherapy is often formulated as a control-affine optimal control problem, which generically yields boundary-selected (bang--bang-type) protocols unless singular arcs occur. This structure, while convenient, neglects saturating pharmacodynamics at high dose rates. We incorporate Michaelis--Menten saturation directly into the therapy channels, making the dynamics non-control-affine and the Hamiltonian nonlinear in each control. With strictly positive exposure penalties, the Hamiltonian admits an explicit three-regime pointwise minimization law: each input is chosen at the lower bound, the upper bound, or as a unique interior minimizer. On interior intervals the strict Legendre condition holds, so continuously modulated dosing arises as a regular interior extremal rather than a singular-arc or smoothing artifact. The resulting transcription is nonconvex; we therefore use a hybrid \emph{Quantum Particle Swarm Optimization} (QPSO)--\emph{Sequential Quadratic Programming} (SQP) pipeline, where QPSO provides a constraint-aware warm start and SQP enforces feasibility and \emph{Karush--Kuhn--Tucker} (KKT) optimality on a collocation grid. Costate reconstruction corroborates the \emph{Pontryagin Minimum Principle} (PMP) structure and the predicted boundary/interior regime transitions.
\end{abstract}

%%%%%%%%%%%%%%%%%%%%%%%%%%%%%%%%%%%%%%%%%%%%%%%%%%%%%%
%                   4. BODY
%%%%%%%%%%%%%%%%%%%%%%%%%%%%%%%%%%%%%%%%%%%%%%%%%%%%%%

\section{Introduction}
\label{sec:intro}

Cancer remains one of the most complex challenges in modern medicine, driving a persistent search for therapeutic strategies that maximize tumor regression while strictly limiting systemic toxicity~\cite{WHOCancerFactSheet,Hanahan2000}. While localized malignancies are often manageable via surgery and radiation, metastatic disease fundamentally relies on systemic intervention~\cite{NCISystemicTherapy}. For decades, the dominant paradigm in clinical chemotherapy has been the administration of the maximum tolerated dose (MTD)~\cite{Romiti2013}. This approach is historically rooted in the linear ``log-kill'' hypothesis proposed by Skipper \textit{et al.}, which posits that the fractional rate of tumor cell kill is directly proportional to instantaneous drug exposure~\cite{Skipper1986}. Although mathematically convenient, aggressive MTD regimens are frequently compromised by a narrow therapeutic index, severe host toxicity, and the rapid emergence of multi-drug resistance~\cite{Romiti2013}. Motivated by these limitations, ``metronomic'' chemotherapy---continuous, lower-dose administration---has received increasing attention as a strategy that can improve long-horizon control by preserving immune competence and, in anti-angiogenic settings, exploiting mechanisms not captured by standard linear dose--response models~\cite{Romiti2013,Ledzewicz2007}.

These limitations have catalyzed the paradigm shift toward immunotherapy, which harnesses the patient's immune system to eradicate malignant cells via mechanisms such as checkpoint blockade and chimeric antigen receptor (CAR) T-cell therapies~\cite{Ribas2018,June2018}. A critical frontier in oncology lies in the synergistic combination of cytotoxic and immune-mediated interventions. It is now well established that chemotherapy-induced apoptosis can be immunogenic: tumor lysis releases tumor-associated antigens that promote immune recognition and priming~\cite{Apetoh2007,Serre2016}. Yet, synergy is governed by a delicate dynamic trade-off, since many cytotoxic agents are inherently immunosuppressive; overly aggressive cytoreduction can impair the immune competence required for durable remission~\cite{Ledzewicz2013,DePillis2006}. This interplay motivates principled scheduling strategies that explicitly balance short-term cytotoxic efficacy against longer-term immune-mediated control.

From a mathematical perspective, chemotherapy scheduling has been extensively studied using deterministic optimal control frameworks. Early models focused on tumor dynamics alone, treating drug administration as a single control subject to toxicity constraints~\cite{Martin1992,Swan1984}. Subsequent work incorporated host dynamics and resistance effects, yielding protocols that often favor MTD-like or bang--bang structures under control-affine (linear-in-dose) pharmacodynamic assumptions~\cite{Ledzewicz2004,Schattler2015}. In parallel, immunotherapy modeling has introduced coupled tumor--immune systems capturing effector activation, exhaustion, and immune-mediated cytotoxicity~\cite{DePillis2006,Kuznetsov1994,Kuznetsov2020}. Several studies have extended these models to combination chemo--immunotherapy, demonstrating that properly timed cytotoxic treatment can enhance immune recognition through antigen release~\cite{Serre2016,Ledzewicz2013}. However, many formulations retain linear dose--response relationships for both chemotherapy and immune stimulation. As a consequence, Pontryagin's Minimum Principle (PMP) typically yields boundary-selected protocols (bang--bang) unless additional singular-arc conditions occur or explicit smoothing/regularization terms are introduced, limiting the ability of the resulting schedules to represent graded, adaptive dosing under physiological saturation.

Despite these advances, a persistent structural limitation remains in many optimal control formulations of chemo--immunotherapy: pharmacodynamic effects are commonly modeled in a control-affine form. In such settings, optimal solutions are generically boundary-driven, and interior modulation is not structurally enforced. In this work, we relax the control-affine assumption by embedding saturating Michaelis--Menten pharmacodynamics directly into the coupled tumor--immune dynamics. When combined with strictly positive exposure penalties, this formulation changes the pointwise Hamiltonian minimization structure and admits regular interior minimizers on intervals where the optimal inputs lie strictly within their bounds. Consequently, continuously modulated titration segments arise as direct implications of first-order optimality, rather than as artifacts of ad hoc smoothing or as singular-control phenomena.

The rational Michaelis--Menten control map renders the dynamics nonlinear in the inputs and, after direct transcription, produces a large-scale NLP with a strongly nonconvex, multi-modal landscape. In this regime, purely gradient-based solvers can be highly initialization-sensitive and may converge to poor local minima, such as trivial mono-therapy schedules or solutions with premature boundary switching~\cite{Betts2009,vonStryk1992}. To improve reliability, we adopt a hybrid global--local pipeline: a QPSO warm-start stage performs stochastic exploration and locates a favorable basin of attraction~\cite{Sun2004a,Wang2018}, and a subsequent SQP refinement stage solves the collocation transcription to tight feasibility and KKT tolerances~\cite{Boggs1995,Nocedal2006}. This hybrid QPSO--SQP framework follows our earlier implementation for optimal therapy scheduling~\cite{KidaneCDC2025} and is essential here for consistently recovering the mixed boundary--interior arc structures reported in Section~\ref{sec:results}.

It is important to delineate the scope of the proposed model relative to the broader challenges of oncology. While acquired resistance is a critical driver of long-term treatment failure, the present study targets a different and more foundational issue: the \emph{optimal-control structure} implied by common pharmacodynamic assumptions. Accordingly, we isolate the effect of saturating (Michaelis--Menten) pharmacodynamics within a homogeneous tumor population and use this setting as a rigorous baseline for \emph{threshold-reaching} therapy design---i.e., reaching a prescribed terminal tumor target with minimal exposure and treatment duration. This baseline can later be extended to incorporate resistant subpopulations and evolutionary selection.

In summary, the contributions of this paper are threefold: (i) an optimal control formulation for chemo--immunotherapy with saturating Michaelis--Menten pharmacodynamics, together with a PMP-based structural characterization showing that, with strictly positive exposure penalties, the Hamiltonian becomes nonlinear and admits regular interior minimizers, so clinically interpretable continuous titration arises as a \emph{regular interior-arc} solution rather than a singular-control phenomenon or numerical smoothing artifact; (ii) a controlled set of numerical case studies (balanced, asymmetric ablation, and the minimum-time limit) that isolate the causal role of exposure regularization in selecting boundary versus interior arc structure while enforcing a hard terminal tumor threshold; and (iii) a hybrid QPSO--SQP numerical pipeline that mitigates local-minimum sensitivity in the resulting nonconvex NLP and achieves high-accuracy feasibility and KKT-optimality for the terminal tumor constraint.

The remainder of this paper is organized as follows. Section~\ref{sec:model_ocp} presents the mathematical model with saturation kinetics and formulates the optimal control problem. Section~\ref{sec:necessary_optimality} derives the necessary conditions for optimality and characterizes the Hamiltonian structure with respect to the controls. The hybrid numerical method is described in Section~\ref{sec:numerical_method}. Section~\ref{sec:results} reports numerical results and controlled case studies, and Section~\ref{sec:conclusion} concludes with limitations and directions for future work.

\section{Mathematical model and optimal control problem formulation}
\label{sec:model_ocp}

\subsection{Tumor--immune dynamics with saturation kinetics}
\label{subsec:model_saturation}

To investigate optimal chemo--immunotherapy schedules, we adopt a three-dimensional tumor--immune interaction model with state variables: tumor burden $x(t)$, immunocompetent effector density $y(t)$, and tumor-associated antigen $z(t)$. The coupling structure follows established formulations in the literature, but we introduce a key modification in the therapy pharmacodynamics.

Classical optimal-control formulations for chemotherapy typically rely on the log-kill hypothesis, in which cytotoxic effect is linear in drug exposure, yielding a control-affine system. In contrast, receptor occupancy, downstream signaling, and other capacity-limited pharmacologic processes impose saturable limits on therapeutic efficacy. To reflect this pharmacological constraint, we replace linear exposure terms with normalized Michaelis--Menten/E$_{\max}$-type saturation laws~\cite{Michaelis1913,Felmlee2012}. In mathematical oncology, saturation functions of Michaelis--Menten or sigmoid type are widely regarded as more realistic than purely linear dose--response laws for anticancer therapy~\cite{Swierniak2009}. Moreover, in constructing a tractable yet interpretable treatment model, we seek to retain the key biological actors and mechanisms while avoiding unnecessary complexity that cannot be supported by available data. Accordingly, we use saturating therapy maps within a low-dimensional coupled chemo--immunotherapy tumor--immune system.

\paragraph{\emph{Units, bounds, and normalization.}}
Time $t$ is measured in \emph{days}. The states $x$ and $z$ are reported in scaled units (with $1$ representing $10^6$ cells), while $y$ is dimensionless (relative to a homeostatic baseline). The control inputs satisfy
\[
u(t)\in[0,u_{\max}], \qquad v(t)\in[0,v_{\max}],
\]
and represent \emph{normalized} dose rates (fractions of the maximum admissible infusion rates). The corresponding pharmacodynamic action is modeled through saturating response functions
\[
\frac{u}{1+u},\qquad \frac{v}{1+v},
\]
which are bounded and strictly increasing on the admissible control set; the associated physical dose-rate scaling is absorbed into the efficacy coefficients.\footnote{Equivalently, one may introduce explicit half-saturation constants $K_u, K_v>0$ and use $u/(K_u+u)$ and $v/(K_v+v)$. In this work, we adopt the normalized form with $K_u=K_v=1$ for notational simplicity.}

\paragraph{Relation to the control-affine baseline.}
This work is inspired by the bang--bang optimal control analysis of Ledzewicz, Maurer and Sch\"attler~\cite{Ledzewicz2024}, which treats a control-affine chemo--immunotherapy model and reports extremals with explicit switching times and terminal time $T$. Here, we modify \emph{only} the therapy channels by replacing linear exposure terms with saturating Michaelis--Menten kinetics; therefore, the schedules and terminal times reported in this paper are not reproductions of~\cite{Ledzewicz2024}, but optimal solutions for the modified pharmacodynamics under the same scaled-unit convention.

\paragraph{\emph{Dynamics.}}
The resulting system is
\begin{align}
    \dot{x} &= \underbrace{\xi x \left( 1 - \frac{x}{x_\infty} \right)}_{\text{Logistic growth}}
    \; - \; \underbrace{\theta x y}_{\text{Immune predation}}
    \; - \; \underbrace{\frac{\alpha x u}{1+u}}_{\text{Saturable cytotoxicity}},
    \label{eq:dx_saturation}\\[0.5em]
    \dot{y} &= \underbrace{a(1-bx)yz}_{\text{Stimulation/suppression}}
    \; + \; \underbrace{\gamma - \delta y}_{\text{Natural homeostasis}}
    \; - \; \underbrace{\frac{\kappa y u}{1+u}}_{\text{Drug toxicity}}
    \; + \; \underbrace{\frac{\nu y v}{1+v}}_{\text{Immunotherapy}},
    \label{eq:dy_saturation}\\[0.5em]
    \dot{z} &= \underbrace{\sigma x}_{\text{Shedding}}
    \; + \; \underbrace{\frac{\psi x u}{1+u}}_{\text{Therapy-induced release}}
    \; - \; \underbrace{\mu z}_{\text{Clearance}}.
    \label{eq:dz_saturation}
\end{align}

Equation~\eqref{eq:dx_saturation} governs the tumor burden $x(t)$. Logistic growth is parameterized by intrinsic growth rate $\xi$ and carrying capacity $x_\infty$, while immune-mediated killing is modeled via the mass-action term $-\theta x y$. The chemotherapy effect is represented by the saturable term $-\alpha x \frac{u}{1+u}$, where $\alpha$ denotes the asymptotic maximum cytotoxic rate under normalized dosing.

Equation~\eqref{eq:dy_saturation} describes the effector population $y(t)$. Antigen $z(t)$ stimulates immune proliferation through $a(1-bx)yz$, with $(1-bx)$ capturing tumor-induced immunosuppression. The terms $\gamma-\delta y$ model natural immune influx and decay. Chemotherapy imposes off-target toxicity via $-\kappa y \frac{u}{1+u}$. Here, immunotherapy is interpreted generically as a stimulatory input acting on the extant effector population and is modeled as proliferation and activation of existing effectors through $+\nu y \frac{v}{1+v}$; the multiplicative dependence on $y$ reflects that efficacy requires a nonzero baseline immune population. The saturating factor $\frac{v}{1+v}$ is interpreted as a bounded stimulatory signal: increasing immunotherapy dose enhances effector activation and proliferation, but only up to a finite limit set by receptor occupancy and downstream signaling capacity. This interpretation is consistent with immunotherapy pharmacology in which exposure--response relationships can plateau as target engagement approaches saturation~\cite{Agrawal2016}, and with tumor--immune modeling practice in which immune recruitment and activation terms are commonly represented using saturating response functions~\cite{Mahlbacher2019}.

Equation~\eqref{eq:dz_saturation} tracks tumor-associated antigen $z(t)$, which couples cytotoxic therapy to immune recruitment. Antigen increases via intrinsic shedding $\sigma x$ and therapy-induced release $\psi x \frac{u}{1+u}$, and decays at rate $\mu z$. This source term represents the immunogenic component of chemotherapy-induced cell death. The optimal-control consequences of the saturating therapy map are derived in Section~\ref{sec:necessary_optimality}.

% ---------------------------------------------------------------
%  TABLE 1 — Variables and parameters
%  FIX: replaced {l l c l c} raw tabular with tabularx + \small
%       to guarantee the table fits inside the 5.0 in text width.
% ---------------------------------------------------------------
\begin{table}[!htbp]
    \centering
    \small
    \caption{Variables and parameters. Numerical values are based on a mouse
      model~\cite{Kuznetsov1994} and are reported in scaled units
      (i.e., $1$ represents $10^6$ cells where indicated).
      Following the normalization convention of~\cite{Ledzewicz2024},
      rate parameters are reported as non-dimensional quantities in the
      scaled time unit (one time unit $\approx$ one day).
      Drug efficacy parameters are illustrative under normalized dosing.}
    \label{tab:parameters}
    \begin{tabularx}{\linewidth}{@{}lLccc@{}}
        \toprule
        Symbol & Interpretation & Value & Dimension & Ref. \\
        \midrule
        $x$ & Tumor volume & -- & $10^6$ cells & \cite{Stepanova1980} \\
        $x_\infty$ & Tumor carrying capacity & 780 & $10^6$ cells & -- \\
        $y$ & Immunocompetent effector density & -- & non-dim. & \cite{Stepanova1980} \\
        $z$ & Tumor antigen & -- & $10^6$ cells & \cite{Serre2016} \\
        \midrule
        $x_0$ & Initial tumor volume & 650 & $10^6$ cells & -- \\
        $y_0$ & Initial effector density & 0.1 & non-dim. & -- \\
        $z_0$ & Initial antigen & 400 & $10^6$ cells & -- \\
        \midrule
        $u$ & Chemo dose rate (normalized) & $[0,1]$ & non-dim. & -- \\
        $v$ & Immuno dose rate (normalized) & $[0,1]$ & non-dim. & -- \\
        \midrule
        $\xi$ & Tumor growth rate & 0.5618 & non-dim. & \cite{Kuznetsov1994} \\
        $\theta$ & Tumor--immune interaction rate & 1.0 & non-dim. & \cite{Kuznetsov1994} \\
        $a$ & Antigen-stimulated prolif.\ rate & 0.00726 & non-dim. & \cite{Ledzewicz2024} \\
        $b$ & Inverse threshold for suppression & 0.00264 & non-dim. & \cite{Ledzewicz2024} \\
        $\gamma$ & Immune influx rate & 0.1181 & non-dim. & \cite{Ledzewicz2024} \\
        $\delta$ & T-cell death rate & 0.3745 & non-dim. & \cite{Ledzewicz2024} \\
        $\sigma$ & Intrinsic immunogenicity & 0.1 & non-dim. & \cite{Ledzewicz2024} \\
        $\mu$ & Antigen clearance rate & 0.45 & non-dim. & \cite{Serre2016} \\
        \midrule
        $\alpha$ & Chemo killing-rate scale on $x$ & 1 & non-dim. & -- \\
        $\kappa$ & Chemo toxicity-rate scale on $y$ & 2 & non-dim. & -- \\
        $\psi$ & Therapy-induced antigen release & 2 & non-dim. & -- \\
        $\nu$ & Immuno efficacy-rate scale on $y$ & 1 & non-dim. & -- \\
        \bottomrule
    \end{tabularx}

    \vspace{0.2cm}
    \footnotesize{$^*$Controls are normalized ($u,v\in[0,1]$). Consistent
      with~\cite{Ledzewicz2024}, all rate-like coefficients are
      non-dimensional under the scaled time unit (one unit $\approx$ one day);
      physical unit conversion is recovered via $\tau=t/t_0$, $t_0=1$ day.}
\end{table}

\subsection{Optimal control problem formulation}

The primary objective in this study is \emph{threshold-reaching} therapy design: to drive the tumor burden to a prescribed terminal target $x_f$ while minimizing cumulative drug exposure and treatment duration (and, when desired, additional state penalties reflecting immune preservation or antigen burden). Accordingly, we formulate the optimal control problem \textbf{(OC)} as follows:

\begin{equation}
\tag{OC}
\label{eq:OC_formulation}
\begin{aligned}
& \underset{u(\cdot), v(\cdot), T}{\text{minimize}}
& &J = \int_0^{T} \left( \ell_1 x(t) + \ell_2 y(t) + \ell_3 z(t) + A u(t) + B v(t) + C \right) dt \\
& \text{subject to}
& & \dot{x}(t) = \xi x\left(1 - \frac{x}{x_\infty}\right) - \theta xy - \frac{\alpha x u}{1 + u}, \\
& & & \dot{y}(t) = a(1 - bx)yz + \gamma - \delta y - \frac{\kappa y u}{1 + u} + \frac{\nu y v}{1 + v}, \\
& & & \dot{z}(t) = \sigma x + \frac{\psi x u}{1 + u} - \mu z, \\
& & & \bm{w}(0) = \bm{w}_0, \quad \text{(Initial conditions)} \\
& & & x(T) = x_f, \quad \text{(Terminal constraint)} \\
& & & 0 \le u(t) \le u_{\max}, \quad 0 \le v(t) \le v_{\max}. \quad \text{(Control bounds)}
\end{aligned}
\end{equation}

In the objective functional $J$, the weights $A$ and $B$ penalize cumulative drug exposure, representing toxicity and cost. The term $C$ penalizes treatment duration, promoting time optimality. The coefficients $\ell_1, \ell_2, \ell_3$ weight the state variables, determining the relative priority of tumor reduction versus preservation of immune effector density and antigen clearance within the composite performance index.

\noindent \textit{Existence of solutions.}
For any measurable controls $(u(\cdot),v(\cdot))\in\Omega$, the right-hand side $F(\bm w,u,v)$ is continuous in $(\bm w,u,v)$ and locally Lipschitz in $\bm w$; hence the state system admits a unique absolutely continuous trajectory on $[0,T]$ for each fixed $T>0$. Under the imposed bounds and forward invariance (nonnegativity) of the states, trajectories remain bounded on the horizon, i.e., confined to a compact subset of the state space.

Regarding existence of an \emph{optimal} control, we appeal to standard existence results for Bolza problems in the \emph{relaxed-control} setting (e.g., Filippov-type existence theorems)~\cite{Filippov1962}, which guarantee existence under compact control bounds and continuity of $(F,L)$ in $(\bm w,u,v)$, without requiring a control-affine structure. In particular, since $\Omega=[0,u_{\max}]\times[0,v_{\max}]$ is nonempty and compact and the running cost is continuous and bounded below, an optimal relaxed control exists; moreover, in our numerical study we compute ordinary (non-relaxed) optimal controls that satisfy the terminal constraint to the stated tolerance.

\section{Optimality conditions and control characterization}
\label{sec:necessary_optimality}

Having formulated the optimal control problem in Section~\ref{sec:model_ocp}, we now characterize the structure of optimal therapeutic schedules using Pontryagin's Minimum Principle (PMP)~\cite{Pontryagin1962,Bressan2007}. PMP provides first-order necessary conditions that couple the optimal state trajectory $\bm{w}^*(t)=[x^*(t),y^*(t),z^*(t)]^T$ to a costate vector $\bm{\lambda}^*(t)\in\mathbb{R}^3$ and a pointwise Hamiltonian minimization law for the controls. Specifically, an optimal pair $(u^*(\cdot), v^*(\cdot))\in\Omega$ must admit associated trajectories $(\bm{w}^*(\cdot), \bm{\lambda}^*(\cdot))$ such that (i) the state equations hold with $\bm{w}(0)=\bm{w}_0$; (ii) the adjoint equations satisfy $\dot{\bm{\lambda}}(t)=-\nabla_{\bm{w}}H(\bm{w}(t),\bm{\lambda}(t),u(t),v(t))$ almost everywhere; (iii) the pointwise minimization condition holds:
\[
(u^*(t),v^*(t))\in\arg\min_{(u,v)\in\Omega} H(\bm{w}^*(t),\bm{\lambda}^*(t),u,v)
\quad\text{for a.e. }t\in[0,T];
\]
and (iv) the appropriate transversality conditions are satisfied for the free-final-time and mixed terminal constraints. The step-by-step derivation of the adjoint system and the algebra leading to the closed-form control characterizations are provided in Appendix~\ref{app:adjoint}--\ref{app:arcs}.

To make these conditions explicit for the present problem, we introduce the Hamiltonian
\begin{equation}
    H(\bm{w}, \bm{\lambda}, u, v) = L(\bm{w}, u, v) + \bm{\lambda}^\top F(\bm{w}, u, v).
    \label{eq:Hamiltonian}
\end{equation}

\noindent \textit{Transversality conditions.}
For the fixed terminal constraint $x(T)=x_f$ and free terminal values $y(T),z(T)$, the transversality conditions are
\[
\lambda_2(T)=0,\qquad \lambda_3(T)=0,
\]
while $\lambda_1(T)$ is free and is determined implicitly by the free-final-time condition (equivalently, $H(T)=0$ for the autonomous Bolza problem).

\subsection{Characterization of optimal controls}
We now characterize the pointwise minimizers of the Hamiltonian with respect to the control variables. The optimal controls $(u^*,v^*)$ minimize the Hamiltonian pointwise over the admissible set
\[
\Omega=[0,u_{\max}]\times[0,v_{\max}].
\]

Building on control characterizations for bounded cancer therapy optimal control problems~\cite{Ledzewicz2004,Schattler2015}, we define the switching potentials
\[
\Phi_u(t) \triangleq -\lambda_1(t)\,\alpha x(t) - \lambda_2(t)\,\kappa y(t) + \lambda_3(t)\,\psi x(t),
\qquad
\Phi_v(t) \triangleq \lambda_2(t)\,\nu y(t).
\]

Differentiating the Hamiltonian with respect to the controls yields
\begin{equation}
    \frac{\partial H}{\partial u} = A + \frac{\Phi_u}{(1+u)^2},
    \qquad
    \frac{\partial H}{\partial v} = B + \frac{\Phi_v}{(1+v)^2}.
    \label{eq:stationarity_partials}
\end{equation}

\paragraph{\emph{Implications for optimal-control structure.}}
Equations~\eqref{eq:stationarity_partials} already reveal the key structural departure from the control-affine setting. Because the therapy enters through rational nonlinearities, the dynamics are no longer control-affine. This changes the PMP minimization structure in an essential way. In particular, on any interval where an optimal control lies strictly inside its bounds, the stationarity condition together with~\eqref{eq:stationarity_partials} implies that the corresponding switching potential must be negative whenever the associated exposure weight is positive. Consequently, interior segments satisfy the strict Legendre condition and therefore correspond to \emph{regular interior arcs} rather than singular-control phenomena. This observation becomes precise by examining the curvature of the Hamiltonian in each control channel and the resulting three-regime pointwise minimization law.

\noindent \textit{Conditional convexity and interior minimizers.}
A crucial distinction from linear-affine control models is that the curvature of $H$ in each control channel is \emph{sign-indefinite} and governed by the switching potentials. The second derivatives are
\begin{equation}
    \frac{\partial^2 H}{\partial u^2} = \frac{-2\Phi_u}{(1+u)^3},
    \qquad
    \frac{\partial^2 H}{\partial v^2} = \frac{-2\Phi_v}{(1+v)^3}.
    \label{eq:convexity_proof}
\end{equation}

Thus, for fixed $(\bm{w},\bm{\lambda})$, the Hamiltonian is strictly convex in $u$ when $\Phi_u(t)<0$ and strictly concave when $\Phi_u(t)>0$; similarly, it is strictly convex in $v$ when $\Phi_v(t)<0$ and strictly concave when $\Phi_v(t)>0$. In particular, if $A>0$ and an optimal control lies strictly inside its bounds, $0<u<u_{\max}$, then the stationarity condition $\partial H/\partial u=0$ implies
\[
\Phi_u(t)=-A(1+u(t))^2<0,
\]
and substituting into~\eqref{eq:convexity_proof} yields $\partial^2 H/\partial u^2>0$ on that interval, i.e., the strict Legendre condition holds in the $u$-channel. Therefore, any interior stationary solution is a unique local minimizer of the Hamiltonian with respect to that control.

The remaining cases correspond to loss of interior stationarity, in which the minimizer is selected at the admissible bounds according to threshold conditions. Specifically, when $\Phi_u(t)\ge -A$, the Hamiltonian is nondecreasing in $u$ on $[0,u_{\max}]$, and the minimizer is attained at the lower bound $u^*(t)=0$; when $\Phi_u(t)\le -A(1+u_{\max})^2$, the Hamiltonian is nonincreasing on $[0,u_{\max}]$, and the minimizer is attained at the upper bound $u^*(t)=u_{\max}$. The same conclusions hold for the $v$-channel with $(B,\Phi_v,v_{\max})$ in place of $(A,\Phi_u,u_{\max})$.

\paragraph{Structural contrast: control-affine vs.\ saturating pharmacodynamics.}
These interior and threshold conditions show that the novelty of the present result is structural rather than merely algebraic. In the classical control-affine setting with linear exposure cost,
\[
\dot{\bm w}=f(\bm w)+g_u(\bm w)u+g_v(\bm w)v,
\qquad
L(\bm w,u,v)=\ell(\bm w)+Au+Bv,
\]
the Hamiltonian is affine in each control channel. Consequently, unless a singular-arc condition holds, PMP selects boundary controls (bang--bang-type behavior) by a sign rule on the switching functions. In contrast, with saturating pharmacodynamics $u\mapsto u/(1+u)$ and $v\mapsto v/(1+v)$, the Hamiltonian becomes \emph{nonlinear} in each channel even though the exposure penalty remains linear. This nonlinearity makes linear exposure regularization sufficient to generate \emph{regular interior minimizers} on intervals where the switching potentials fall in an interior regime. Thus, continuously modulated dosing arises as a first-order PMP consequence of saturation--regularization interaction, without introducing quadratic smoothing terms or invoking singular-control phenomena.

The preceding observations can now be summarized in a precise pointwise minimization law.

\begin{proposition}[Pointwise minimizers and three-regime structure]
\label{prop:pointwise_minimizers}
Fix $t\in[0,T]$ and assume $A>0$, $B>0$ with bounds $u\in[0,u_{\max}]$, $v\in[0,v_{\max}]$.
Define the thresholds
\[
\Phi_u^{(0)} \triangleq -A(1+0)^2=-A,\qquad
\Phi_u^{(\max)} \triangleq -A(1+u_{\max})^2,
\]
\[
\Phi_v^{(0)} \triangleq -B(1+0)^2=-B,\qquad
\Phi_v^{(\max)} \triangleq -B(1+v_{\max})^2.
\]

Then the Hamiltonian minimizers in each channel are
\[
u^*(t)=
\begin{cases}
0, & \Phi_u(t)\ge \Phi_u^{(0)},\\[1mm]
-1+\sqrt{-\Phi_u(t)/A}, & \Phi_u^{(\max)}<\Phi_u(t)<\Phi_u^{(0)},\\[1mm]
u_{\max}, & \Phi_u(t)\le \Phi_u^{(\max)},
\end{cases}
\]
and
\[
v^*(t)=
\begin{cases}
0, & \Phi_v(t)\ge \Phi_v^{(0)},\\[1mm]
-1+\sqrt{-\Phi_v(t)/B}, & \Phi_v^{(\max)}<\Phi_v(t)<\Phi_v^{(0)},\\[1mm]
v_{\max}, & \Phi_v(t)\le \Phi_v^{(\max)}.
\end{cases}
\]

Moreover, on the interior regimes the strict Legendre condition holds in the corresponding channel,
i.e., $\partial^2H/\partial u^2>0$ when $\Phi_u^{(\max)}<\Phi_u<\Phi_u^{(0)}$ and
$\partial^2H/\partial v^2>0$ when $\Phi_v^{(\max)}<\Phi_v<\Phi_v^{(0)}$.
\end{proposition}

\begin{proof}
Consider the $u$-channel. Up to terms independent of $u$, the Hamiltonian dependence on $u$ is
$H(u)=Au+\Phi_u\,\frac{u}{1+u}$ on $[0,u_{\max}]$. Hence,
\[
\frac{dH}{du}=A+\frac{\Phi_u}{(1+u)^2},\qquad
\frac{d^2H}{du^2}=\frac{-2\Phi_u}{(1+u)^3}.
\]

If $\Phi_u\ge -A$, then $\frac{dH}{du}\ge 0$ for all $u\ge 0$ and the minimizer is $u^*=0$.
If $\Phi_u\le -A(1+u_{\max})^2$, then $\frac{dH}{du}\le 0$ on $[0,u_{\max}]$ and the minimizer is $u^*=u_{\max}$.
Otherwise, $\Phi_u\in\big(-A(1+u_{\max})^2,-A\big)$ and there is a unique root of $\frac{dH}{du}=0$ at
$u=-1+\sqrt{-\Phi_u/A}\in(0,u_{\max})$. On this interior regime $\Phi_u<0$, so $\frac{d^2H}{du^2}>0$ and the stationary point is the unique minimizer. The proof for the $v$-channel is identical with $(B,\Phi_v,v_{\max})$ replacing $(A,\Phi_u,u_{\max})$.
\end{proof}

Equivalently, solving $\partial H/\partial u=0$ and $\partial H/\partial v=0$ yields the candidate interior controls, and projecting onto $\Omega$ gives the compact expressions
\begin{align}
    u^*(t) &= \min\left\{u_{\max},\; \max\left\{0,\; -1 + \sqrt{\frac{-\Phi_u(t)}{A}} \right\}\right\},
    \label{eq:optimal_u} \\
    v^*(t) &= \min\left\{v_{\max},\; \max\left\{0,\; -1 + \sqrt{\frac{-\Phi_v(t)}{B}} \right\}\right\}.
    \label{eq:optimal_v}
\end{align}

\section{Hybrid numerical optimization framework}
\label{sec:numerical_method}

Although Section~\ref{sec:necessary_optimality} provides a structural PMP characterization of the optimal controls, computing complete optimal trajectories remains a challenging nonlinear optimization problem. The optimal control problem is strongly nonlinear due to the coupled tumor--immune dynamics and the saturating (Michaelis--Menten) therapy channels. After direct transcription, the continuous-time problem becomes a large nonconvex nonlinear program (NLP). In such multi-modal settings, purely gradient-based solvers can be sensitive to initialization and may converge to suboptimal basins~\cite{Betts2009,vonStryk1992}. To improve robustness, we employ a two-stage hybrid framework that combines global stochastic exploration (QPSO) with deterministic local refinement (SQP), following our earlier hybrid-QPSO approach for therapy scheduling~\cite{KidaneCDC2025}.

\paragraph{\emph{Stage 1: stochastic global search (QPSO)}}
We first apply Quantum Particle Swarm Optimization (QPSO) to obtain a reliable warm start by locating a high-quality basin of attraction near feasibility. The controls are parameterized on a coarse grid of $N\!+\!1$ nodes (with $N=100$), i.e.,
\[
\mathbf{u}=\{u_k\}_{k=0}^{N},\qquad \mathbf{v}=\{v_k\}_{k=0}^{N},\qquad
0\le u_k\le u_{\max},\; 0\le v_k\le v_{\max}.
\]

For each candidate discrete control trajectory, the state dynamics are simulated using \texttt{ode45}, with $(u,v)$ linearly interpolated between nodes. Because QPSO does not enforce hard equality constraints directly, particles are ranked using the augmented objective
\begin{equation}
J_{\mathrm{aug}}(\mathbf{u},\mathbf{v};T)
=
\int_{0}^{T}\!\Big(\ell_1 x + \ell_2 y + \ell_3 z + A u + B v + C\Big)\,dt
+\rho\,\big(x(T)-x_f\big)^2,
\label{eq:J_aug_qpso}
\end{equation}
where the running integral is computed by trapezoidal quadrature on the simulation time grid. The quadratic terminal penalty is introduced only to steer the derivative-free swarm toward the terminal manifold $x(T)=x_f$ and is used solely for warm-start generation (it is not part of the final constrained solve). In our implementation, $\rho=10^{6}$. The QPSO stage uses a swarm size $N_p=50$ and runs for $k_{\max}=100$ iterations. The best particle provides a coarse control trajectory $(\tilde{u},\tilde{v})$ that captures the global structure but may contain particle-induced oscillations. This coarse, globally informed solution is then refined deterministically within a constrained collocation framework.

\paragraph{\emph{Stage 2: deterministic refinement (smoothing, collocation, SQP).}}
The QPSO control trajectory is smoothed using a Gaussian kernel (window width $\sigma_{\mathrm{smooth}}=10$ nodes) to suppress high-frequency oscillations while preserving large-scale structure. This smoothed control is then used to initialize a direct collocation transcription on the same $N\!+\!1$ node grid. Free final time is handled by treating $T$ as a decision variable and using the step size $h=T/N$. Let $\mathbf{w}_k=[x_k,y_k,z_k]^\top$ and $\mathbf{u}_k=[u_k,v_k]^\top$ denote the discrete state and control at node $k$. Using trapezoidal collocation, the defect constraints are
\begin{equation}
\mathbf{w}_{k+1}-\mathbf{w}_k
-\frac{h}{2}\Big(F(\mathbf{w}_k,\mathbf{u}_k)+F(\mathbf{w}_{k+1},\mathbf{u}_{k+1})\Big)=\mathbf{0},
\qquad k=0,\dots,N-1,
\label{eq:trap_defects}
\end{equation}
together with the initial condition $\mathbf{w}_0=\mathbf{w}_0^{\mathrm{given}}$ and the \emph{hard} terminal equality constraint $x_N=x_f$. In addition, we enforce state nonnegativity at all nodes: $x_k\ge0$, $y_k\ge0$, $z_k\ge0$.

The resulting decision vector is
\[
\mathbf{z}
=
\Big(\{\mathbf{w}_k\}_{k=0}^{N},\;\{\mathbf{u}_k\}_{k=0}^{N},\;T\Big),
\]
and the NLP objective is evaluated by trapezoidal quadrature:
\[
J(\mathbf{z})
\approx
\sum_{k=0}^{N-1}\frac{h}{2}\Big(L(\mathbf{w}_k,\mathbf{u}_k)+L(\mathbf{w}_{k+1},\mathbf{u}_{k+1})\Big),
\quad
L=\ell_1 x+\ell_2 y+\ell_3 z + A u + B v + C.
\]

We solve this NLP using MATLAB \texttt{fmincon} with the SQP algorithm. Importantly, the reported solutions solve the original constrained problem: the terminal condition $x_N=x_f$ and the collocation defects are enforced as constraints and are driven below the specified tolerances (\texttt{ConstraintTolerance} $=10^{-6}$ and \texttt{OptimalityTolerance} $=10^{-6}$), up to \texttt{fmincon}'s internal scaling.

\paragraph{\emph{Implementation details and reproducibility.}}
For reproducibility, Table~\ref{tab:solver_settings} lists all numerical settings required to reproduce the reported solutions, including the terminal penalty form/weight, the free-time transcription $h=T/N$, and the exact decision-vector layout. The QPSO stage uses MATLAB's pseudorandom number generator; for exact replication of a run, the RNG seed can be fixed via \texttt{rng(seed)} prior to the QPSO initialization. The QPSO rollout uses \texttt{ode45} with \texttt{RelTol} $=10^{-3}$ and \texttt{AbsTol} $=10^{-3}$, while the SQP refinement uses \texttt{fmincon} with \texttt{MaxIterations}$=5000$ and \texttt{MaxFunctionEvaluations}$=3\times10^{6}$.

% ---------------------------------------------------------------
%  TABLE 2 — Solver settings
%  FIX: replaced {lcc} raw tabular with tabularx + \small
% ---------------------------------------------------------------
\begin{table}[!htbp]
\centering
\small
\caption{Numerical solver parameters and implementation settings (MATLAB).}
\label{tab:solver_settings}
\begin{tabularx}{\linewidth}{@{}LCC@{}}
\toprule
Parameter & Symbol / Option & Value \\
\midrule
\multicolumn{3}{@{}l@{}}{\textit{QPSO (global warm start)}}\\
Control grid nodes & $N+1$ & $101$ ($N=100$) \\
Swarm size & $N_p$ & $50$ \\
Max iterations & $k_{\max}$ & $100$ \\
Terminal penalty form & $J_{\mathrm{aug}}$ & $J+\rho(x(T)-x_f)^2$ \\
Terminal penalty weight & $\rho$ & $10^{6}$ \\
ODE rollout solver & -- & \texttt{ode45} \\
ODE tolerances & \texttt{RelTol}, \texttt{AbsTol} & $10^{-3},\,10^{-3}$ \\
Smoothing kernel & $\sigma_{\mathrm{smooth}}$ & 10 nodes (Gaussian) \\
\midrule
\multicolumn{3}{@{}l@{}}{\textit{Collocation + SQP refinement}}\\
Collocation method & -- & Trapezoidal \\
Time discretization & $h$ & $h=T/N$ (free $T$) \\
Constraints & -- & Defects $+\, x_N=x_f$ $+$ positivity \\
Local solver & -- & \texttt{fmincon} (SQP) \\
Max iterations & \texttt{MaxIterations} & $5000$ \\
Max function evals & \texttt{MaxFunctionEvaluations} & $3\times 10^6$ \\
Tolerance (constraints) & \texttt{ConstraintTolerance} & $10^{-6}$ \\
Tolerance (optimality) & \texttt{OptimalityTolerance} & $10^{-6}$ \\
\bottomrule
\end{tabularx}
\end{table}

\FloatBarrier

\section{Numerical results and discussion}
\label{sec:results}

This section reports the optimal therapeutic schedules computed with the proposed hybrid QPSO--SQP framework. The need for a hybrid strategy is practical rather than cosmetic: after direct transcription, the problem becomes a large nonconvex NLP due to the coupled nonlinear dynamics and the rational Michaelis--Menten control map. In preliminary trials, a standalone SQP method initialized from uninformed guesses frequently converged to suboptimal basins (e.g., static monotherapy profiles or premature boundary switching). The QPSO stage provides global exploration to identify a reliable basin of attraction, while the SQP stage enforces high-accuracy feasibility and optimality on a refined collocation grid.

All simulations use the normalized parameters in Table~\ref{tab:parameters} with initial conditions $x_0=650$, $y_0=0.1$, and $z_0=400$. The terminal condition $x(T)=x_f$ is imposed as a hard equality constraint in the NLP and satisfied to a feasibility tolerance of $10^{-6}$, ensuring that tumor-reduction targets are met strictly rather than asymptotically. Unless otherwise stated, controls are normalized to $(u,v)\in[0,1]$.

\subsection{Balanced objectives and emergence of interior modulation (Case A)}
We first consider the balanced regime with $\ell=(\ell_1,\ell_2,\ell_3)=(0,0,0)$ and exposure/time penalties $(A,B,C)=(1,\,0.3,\,1)$. This case is the primary demonstration of the paper's central claim: in a saturated pharmacodynamic model, exposure penalties produce \emph{regular interior arcs} (continuous modulation) rather than purely boundary-driven bang--bang schedules.

Figure~\ref{fig:case_A_solution} summarizes the optimal state and control trajectories. Panel (a) shows that the tumor burden $x(t)$ decreases monotonically to the enforced terminal target $x(T)=x_f\le 1$. The antigen signal $z(t)$ exhibits a transient rise and fall, while the immune population $y(t)$ remains low early and then increases sharply late in the therapy. Panel (b) shows the corresponding control structure: chemotherapy $u(t)$ begins near the upper bound to drive early cytoreduction, transitions through a continuously varying \emph{titration} interval, and then shuts off; immunotherapy $v(t)$ follows a broad interior pulse.

% FIX: changed width=\textwidth to width=\linewidth for all figures
\begin{figure}[!htbp]
    \centering
    \includegraphics[width=\linewidth,keepaspectratio]{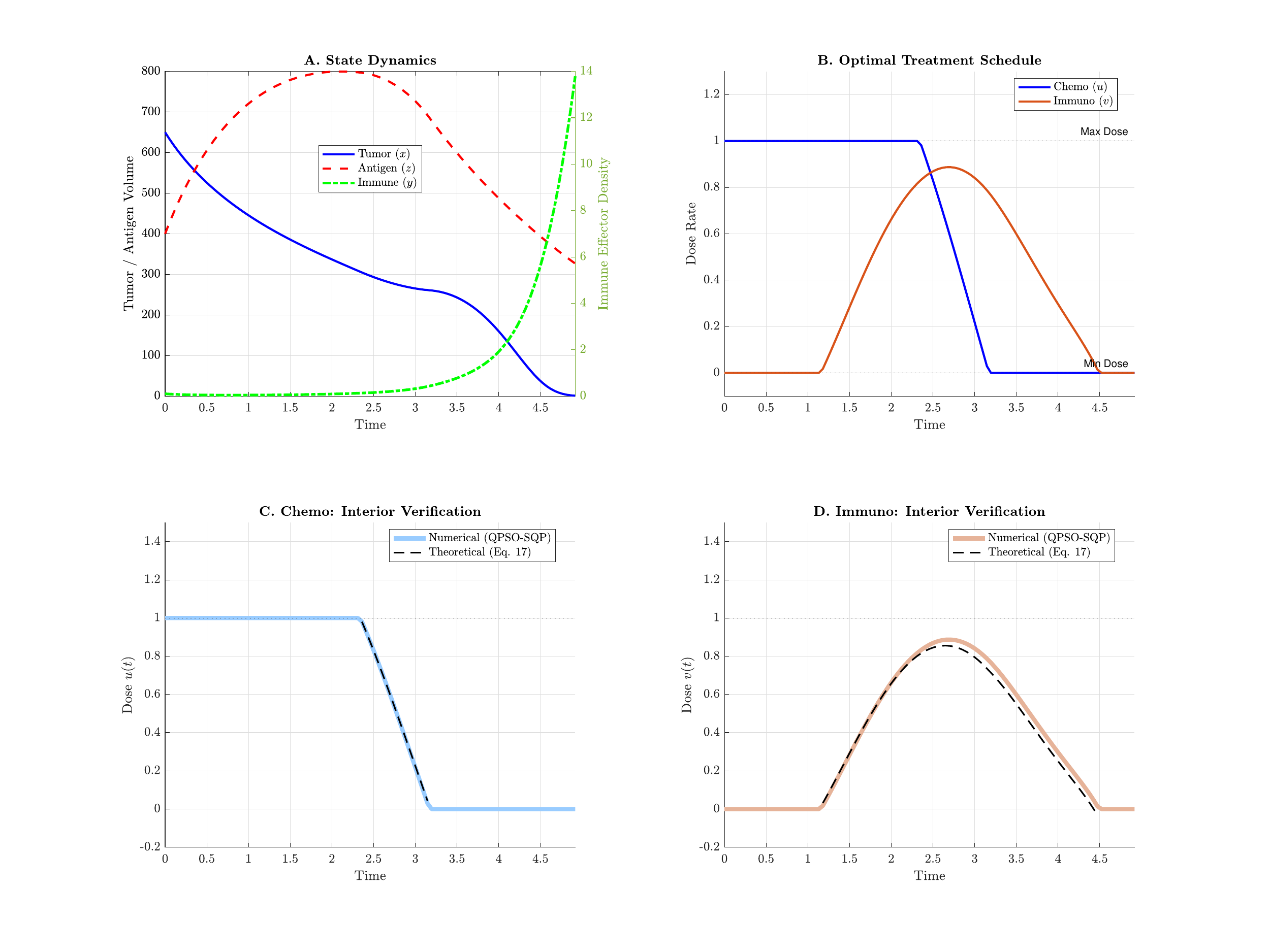}
    \caption{Case A (balanced objectives): $\ell=(0,0,0)$ and $(A,B,C)=(1,0.3,1)$ with free final time. (a) State trajectories showing monotone tumor reduction, a transient antigen response, and late immune amplification. (b) Optimal controls exhibiting a mixed structure: an initial boundary arc for chemotherapy, an interior titration interval, and a broad interior immunotherapy pulse. (c--d) Verification overlays comparing the collocation solution (solid) with the analytical interior-arc feedback law (dashed) on intervals where the numerical control is interior.}
    \label{fig:case_A_solution}
\end{figure}

The presence of continuous titration is a direct consequence of diminishing returns induced by the Michaelis--Menten map $u\mapsto u/(1+u)$. As $u$ increases, the marginal effect on the state dynamics saturates, while the linear exposure term $Au$ increases at a constant rate. The optimizer therefore reduces dosing smoothly once the marginal benefit falls below the marginal exposure cost, producing an interior arc rather than sustained saturation. Quantitatively, the terminal target is reached at $T\approx 4.92$ with cumulative chemotherapy exposure $\int_0^T u(t)\,dt\approx 2.77$ and cumulative immunotherapy exposure $\int_0^T v(t)\,dt\approx 1.79$ (in normalized dose-time units).

Although immune depletion is not directly penalized here ($\ell_2=0$), the optimal strategy does not suppress the immune compartment. Instead, the computed schedule preserves immune competence and leverages the endogenous predation term $-\theta x y$ to complete tumor clearance with reduced reliance on chemotherapy, yielding an emergent chemo-sparing effect consistent with the coupled model structure.

\begin{figure}[!htbp]
    \centering
    % FIX: width=\linewidth instead of width=\textwidth
    \includegraphics[width=\linewidth,keepaspectratio]{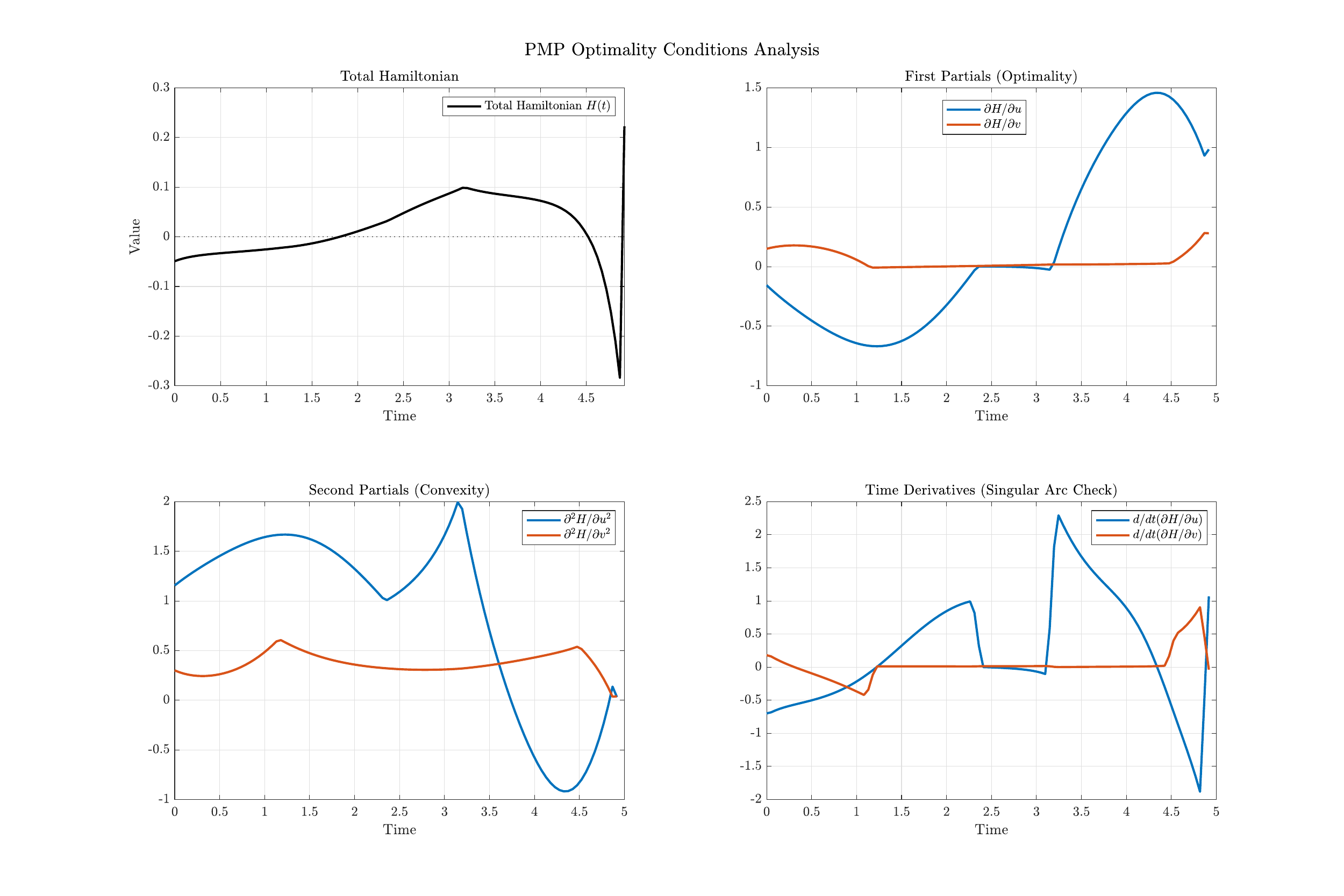}
    \caption{PMP diagnostics for Case A. (a) Hamiltonian trace for the free-final-time autonomous formulation (interpreted as a consistency check). (b) First partials $\partial H/\partial u$ and $\partial H/\partial v$, which vanish only on interior arcs and satisfy sign conditions on boundary arcs. (c) Second partials $\partial^2 H/\partial u^2$ and $\partial^2 H/\partial v^2$; positivity on interior stationary intervals is consistent with the strict Legendre condition. (d) Time derivatives of the first partials, highlighting transitions between boundary and interior regimes.}
    \label{fig:case_A_pmp}
\end{figure}

Additional PMP diagnostics for Case~A are reported in Figure~\ref{fig:case_A_pmp} and are intended as compact consistency checks of the arc classification (boundary vs.\ interior). Panel (a) plots the Hamiltonian trace; for the autonomous free-final-time formulation it should be approximately constant (and satisfy the terminal condition), so large deviations would indicate inconsistency. Panel (b) shows the first partials $\partial H/\partial u$ and $\partial H/\partial v$: these vanish only on interior arcs where $0<u^*(t)<u_{\max}$ (resp.\ $0<v^*(t)<v_{\max}$), while on boundary arcs they satisfy the corresponding sign conditions that select $u^*(t)\in\{0,u_{\max}\}$ and $v^*(t)\in\{0,v_{\max}\}$. Panel (c) reports $\partial^2H/\partial u^2$ and $\partial^2H/\partial v^2$; positivity on interior intervals is consistent with the strict Legendre condition and supports that the continuously modulated segments are regular interior minimizers rather than singular-control behavior. Panel (d) plots the time derivatives of the first partials to make the transitions between boundary and interior regimes visible. The adjoint system and the analytic interior-arc expressions used for the overlays are given in Appendix~\ref{app:adjoint}--\ref{app:arcs}.

\subsection{Mechanism Isolation via asymmetric regularization (Case B)}
\label{subsec:caseB_asym_reg}

To isolate the role of exposure regularization \emph{per control channel}, we consider an asymmetric weighting regime in which the chemotherapy penalty is retained while the immunotherapy penalty is removed. Specifically, we set
\[
\ell_1=\ell_2=\ell_3=0,\qquad A=1,\qquad B=0,\qquad C=1,
\]
and impose the same terminal constraint $x(T)=x_f$. This targeted ablation bridges the fully regularized Case~A and the minimum-time limit Case~C.

\begin{figure}[!htbp]
    \centering
    % FIX: width=\linewidth instead of width=\textwidth
    \includegraphics[width=\linewidth,keepaspectratio]{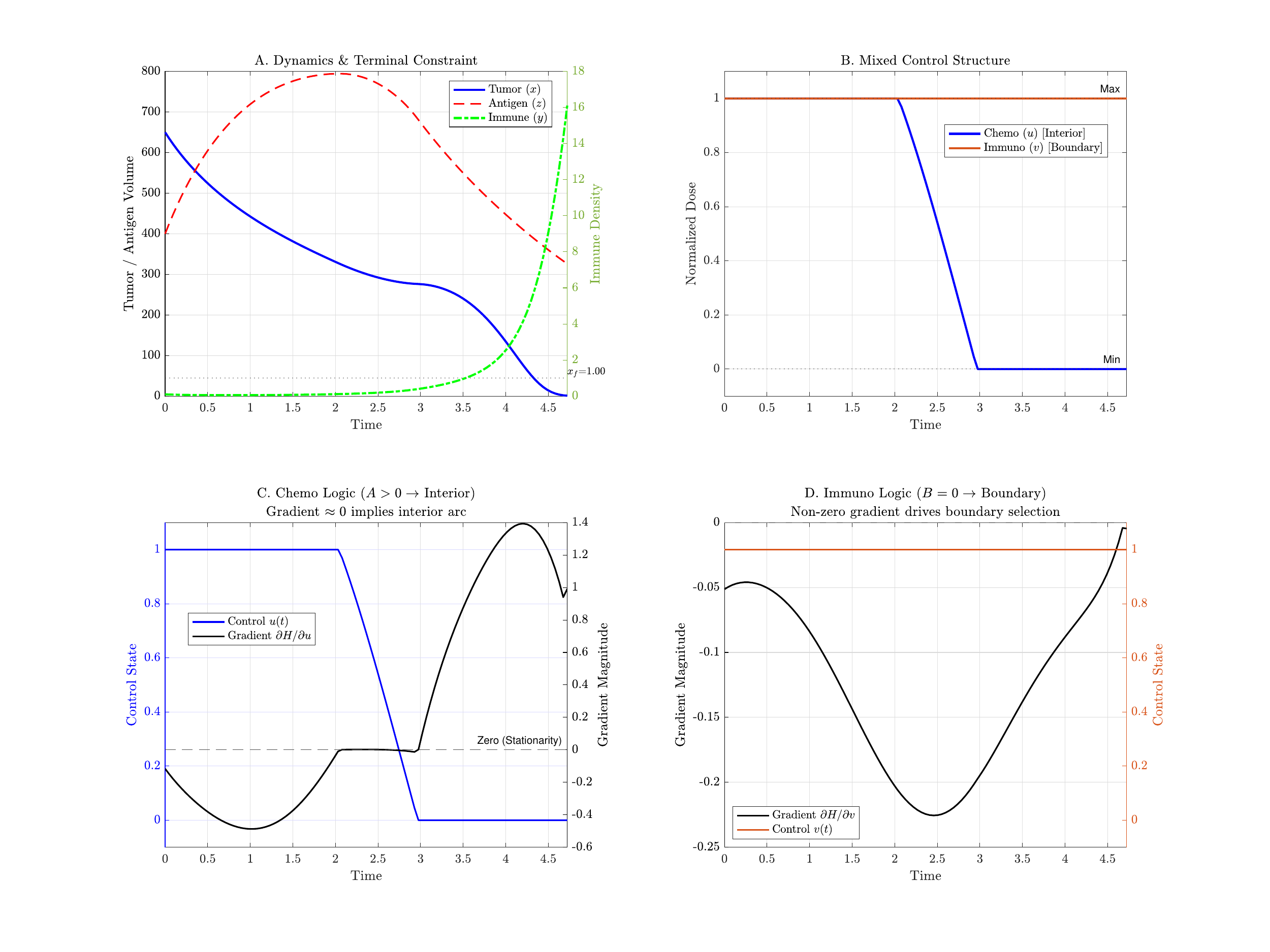}
    \caption{Case B (asymmetric regularization: $A=1$, $B=0$, $C=1$). (A) State trajectories satisfying the terminal tumor constraint. (B) The optimal controls bifurcate by channel: chemotherapy $u(t)$ retains an interior titration segment due to $A>0$, whereas immunotherapy $v(t)$ is boundary-selected and saturates at $v_{\max}$ when $B=0$. (C--D) PMP diagnostics: $\partial H/\partial u$ is pinned to zero on the interior chemotherapy interval (stationarity), while $\partial H/\partial v$ maintains a sign consistent with boundary selection in the unregularized immunotherapy channel.}
    \label{fig:case_B_solution}
\end{figure}

Figure~\ref{fig:case_B_solution} shows that the optimal policy becomes structurally asymmetric. Because $A>0$, the Hamiltonian is regularized in $u$ and admits an interior minimizer; accordingly, the chemotherapy profile retains an interior titration segment with $0<u(t)<u_{\max}$. In this interval, diminishing returns from the Michaelis--Menten map compete with the linear exposure cost, producing a smooth taper from saturation toward shutoff. In contrast, setting $B=0$ removes exposure regularization in the immunotherapy channel. Although the pharmacodynamic term $\frac{v}{1+v}$ remains nonlinear, the Hamiltonian becomes \emph{monotone} with respect to $v$ (i.e., it has no interior turning point), so the minimizer is selected by the PMP sign rule and the immunotherapy input collapses to a boundary arc (here, $v(t)=v_{\max}$ over the active window).

The PMP diagnostics provide a direct explanation for this mixed structural type. On the interior chemotherapy interval (Figure~\ref{fig:case_B_solution}(C)), optimality requires the stationarity condition $\partial H/\partial u = 0$, and the computed stationarity residual is numerically pinned to zero precisely where $u(t)$ lies strictly inside its bounds. By contrast, for the unregularized immunotherapy channel (Figure~\ref{fig:case_B_solution}(D)), optimality is governed by a sign condition rather than stationarity; accordingly, $\partial H/\partial v$ does not settle at zero and instead maintains the sign that selects the observed boundary action.

Case~B therefore strengthens the main causal claim of the paper: interior modulation is not a byproduct of discretization or solver smoothing, but appears selectively in the channels where exposure regularization is present. This confirms that the continuously modulated schedules in Case~A are a structural consequence of regularization interacting with saturating kinetics, and it also demonstrates that the hybrid QPSO--SQP framework can robustly resolve mixed arc types within a single optimal trajectory.

\subsection{Validation via the minimum-time limiting case}
\label{subsec:caseC_min_time}

To isolate the mechanism responsible for the continuously modulated interior arcs in Cases~A--B, we consider the \emph{minimum-time limiting regime} (Case~C) obtained by removing \emph{all} running penalties on state and drug exposure:
\[
\ell_1=\ell_2=\ell_3=0,\qquad A=B=0,\qquad C>0.
\]

With the terminal tumor requirement enforced as a hard constraint, $x(T)=x_f$, the objective reduces to
\[
J=\int_{0}^{T} C\,dt = C\,T,
\]
so the problem becomes a terminal-constrained minimum-time formulation. In this limit, the Hamiltonian remains \emph{nonlinear} in the controls through the saturating maps $u/(1+u)$ and $v/(1+v)$, but it is \emph{unregularized} (non-coercive) with respect to $(u,v)$ because there is no control-exposure cost. Consequently, on intervals where the switching potentials are nonzero, the Hamiltonian is monotone in each control channel and admits no interior stationary minimizer; the PMP condition therefore reduces to a \emph{boundary-selection rule} (bang--bang-type behavior, with isolated switching times).

Figure~\ref{fig:case_C_min_time} summarizes the resulting structure. Panel~(A) shows that the tumor burden decreases monotonically to the prescribed target $x_f$, but the qualitative nature of the optimal control collapses to a boundary strategy. Panel~(B) displays the corresponding inputs: chemotherapy saturates at $u_{\max}$ over an initial arc to achieve the fastest admissible cytoreduction and then switches abruptly to $u(t)=0$; the immunotherapy input likewise occupies an extreme arc over the horizon (here, $v(t)\approx v_{\max}$ in the computed solution), reflecting the same boundary-selection mechanism when $B=0$.

\begin{figure}[!htbp]
    \centering
    % FIX: width=\linewidth instead of width=\textwidth
    \includegraphics[width=\linewidth,keepaspectratio]{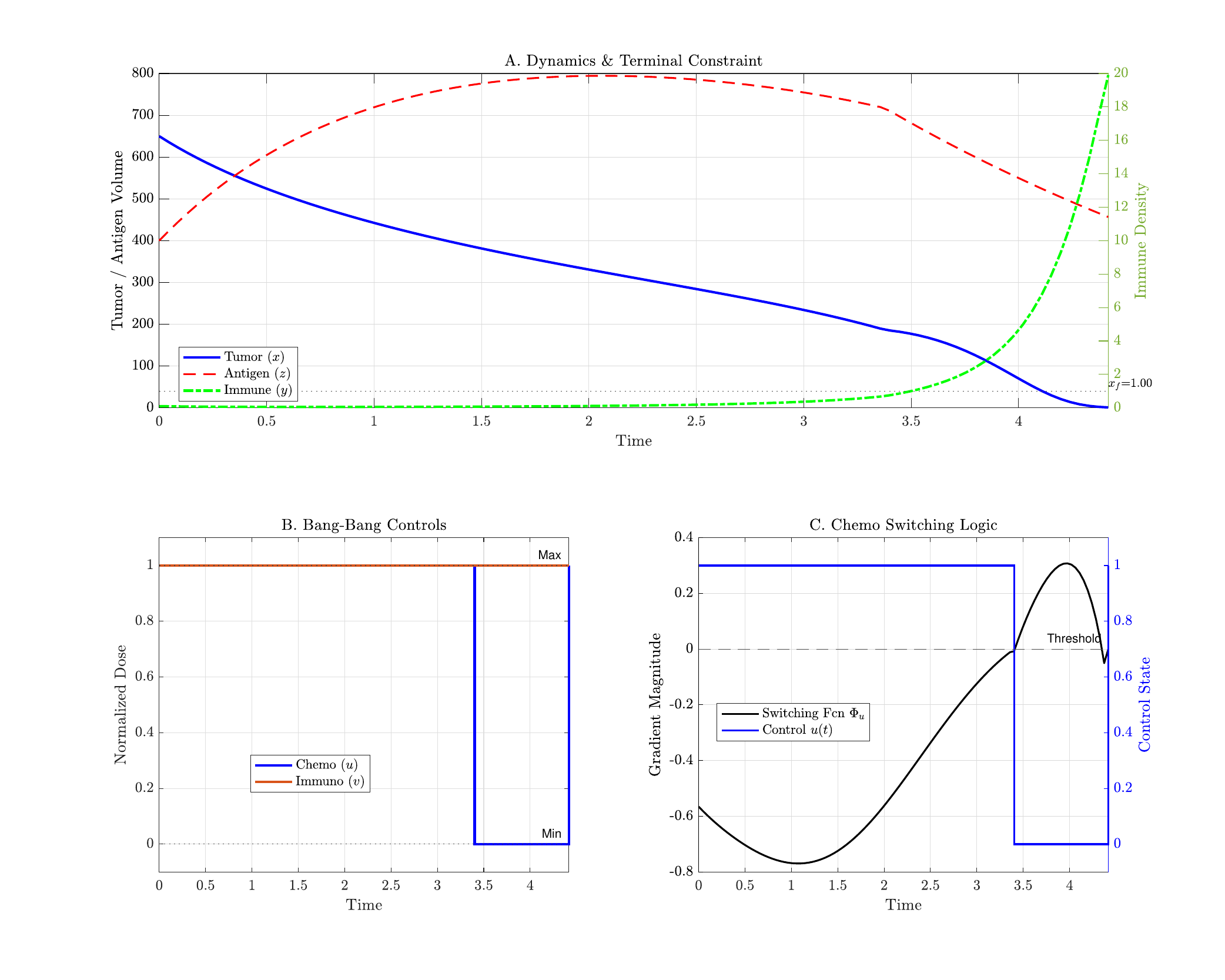}
    \caption{Case~C (minimum-time limit): $\ell_1=\ell_2=\ell_3=A=B=0$, $C>0$.
    \textbf{(A)} State evolution under the minimum-time policy.
    \textbf{(B)} Optimal inputs collapse to boundary arcs; in particular, $u(t)$ is bang--bang.
    \textbf{(C)} PMP switching logic for chemotherapy: the switching function $\Phi_u(t)$ crosses the switching surface $\Phi_u=0$ at $t\approx t_s$, and the bang--bang switch in $u(t)$ occurs at the same instant.}
    \label{fig:case_C_min_time}
\end{figure}

The key verification is shown directly in Figure~\ref{fig:case_C_min_time}(C). The chemotherapy switching function $\Phi_u(t)$ exhibits a \emph{clean transversal} zero-crossing at $t\approx t_s\approx 3.4$, and the control switches at the same instant. The transversal crossing (nonzero slope at the crossing) indicates a \emph{strict switching} event, ruling out a sustained singular interval ($\Phi_u(t)\equiv 0$ on a nontrivial time set) in this regime. Accordingly, the PMP boundary rule selects the admissible extremes based on the sign of the switching function:
\[
u^*(t)=
\begin{cases}
u_{\max}, & \Phi_u(t)<0,\\[2mm]
0, & \Phi_u(t)>0,
\end{cases}
\qquad
(\text{switching at } \Phi_u(t_s)=0),
\]
up to sets of measure zero.

Case~C therefore functions as a rigorous \emph{negative control}, isolating the causal mechanism that determines the control structure. When exposure regularization is removed ($A=B=0$), interior stationary arcs are generically eliminated and the extremal undergoes a structural bifurcation to a boundary-driven, discontinuous (bang--bang-type) policy. By contrast, in Cases~A--B, strictly positive exposure weights $(A,B>0)$ restore coercivity of the Hamiltonian in the controls and enable well-defined interior minimizers; coupled with Michaelis--Menten saturation, this regularization yields continuous dose titration. The comparison confirms that the smooth modulation observed in Cases~A--B is not a discretization artifact of the QPSO--SQP framework, but an intrinsic structural consequence of exposure regularization interacting with saturating pharmacodynamics.

\subsection{Sensitivity to admissible dose bounds and local parameter robustness}
\label{subsec:bound_sensitivity}

\paragraph{\emph{Sensitivity to admissible dose bounds under normalized saturation scaling.}}
A standard concern in bounded optimal control is whether continuously modulated schedules arise only because an upper bound is active. To assess how the optimal structure depends on admissible dose limits in our \emph{normalized} saturation model $u/(1+u)$ and $v/(1+v)$, we re-solve the balanced problem (Case~A) under a sweep of relaxed maxima
\[
u_{\max}=v_{\max}\in\{1,\;1.5,\;2\},
\]
and compare the resulting optimal schedules in normalized time $\tau=t/T\in[0,1]$. Figure~\ref{fig:bound_independence} overlays the chemotherapy and immunotherapy controls across this sweep.

Two observations are consistent across the tested range. First, after time normalization, the control trajectories remain similar in overall shape and in the timing of switching/taper regions. Second, increasing $u_{\max}$ and $v_{\max}$ does not simply push the solution to the new saturation limits: the optimizer continues to spend substantial time in an intermediate transition region of the Michaelis--Menten map where marginal benefit is high relative to linear exposure cost, and the fraction of time spent at the upper bound does not increase monotonically. This supports the interpretation that interior modulation in Case~A is selected by the saturation--regularization trade-off rather than being solely enforced by an active upper-bound constraint.

\noindent \textit{Interpretation under normalization.}
Because we use the normalized form $u/(1+u)$ (equivalently $K_u=1$) and $v/(1+v)$ (equivalently $K_v=1$), changing $u_{\max}$ and $v_{\max}$ also changes how far the admissible range extends beyond the half-saturation point. Thus, this sweep should be viewed as a sensitivity check to admissible dose bounds \emph{under the fixed normalized saturation scaling}, rather than a bound-independence claim that is invariant to the choice of half-saturation constants.

\begin{figure}[!htbp]
    \centering
    % FIX: width=\linewidth instead of width=\textwidth
    \includegraphics[width=\linewidth,keepaspectratio]{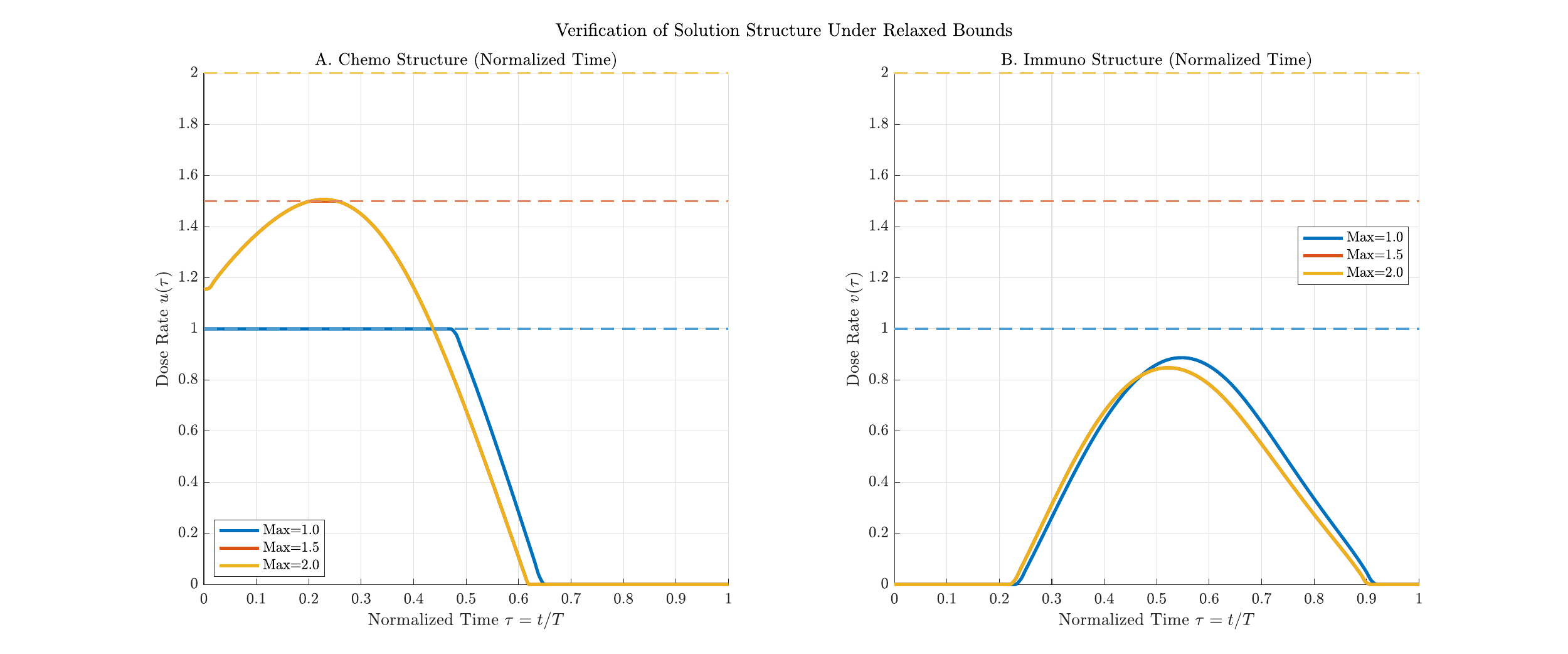}
    \caption{Sensitivity to admissible dose bounds for the balanced objective (Case~A), shown in normalized time $\tau=t/T$. Optimal controls are overlaid for $u_{\max}=v_{\max}\in\{1,1.5,2\}$. The qualitative structure persists under relaxed bounds and does not trivially saturate at the enlarged maxima, consistent with an interior structure selected by saturation interacting with exposure regularization.}
    \label{fig:bound_independence}
\end{figure}

\begin{figure}[!htbp]
    \centering
    % FIX: width=\linewidth instead of width=\textwidth
    \includegraphics[width=\linewidth,keepaspectratio]{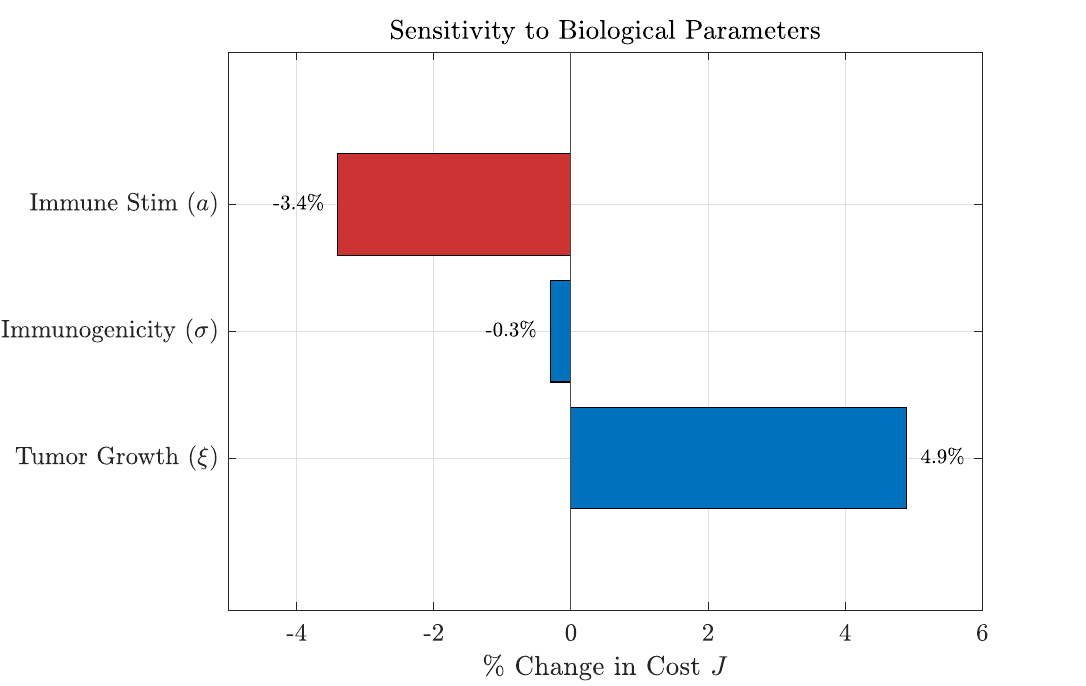}
    \caption{Local parameter sensitivity of the optimal objective value $J^*$ under $\pm 20\%$ one-at-a-time perturbations. Each bar corresponds to \emph{re-optimizing} the OCP for the perturbed parameter set and reporting the relative change in the resulting optimal cost.}
    \label{fig:sensitivity}
\end{figure}

\paragraph{\emph{Local parameter robustness via re-optimization (one-at-a-time perturbations).}}
To quantify robustness to moderate biological uncertainty, we perform a one-at-a-time local sensitivity study in which selected parameters are perturbed by $\pm 20\%$ and the optimal control problem is \emph{re-solved} for each perturbed parameter set using the same QPSO--SQP pipeline and solver tolerances as in the nominal case. That is, for each perturbed parameter $\tilde{p}$ we compute a new optimal solution $(u^*_{\tilde{p}}(\cdot),v^*_{\tilde{p}}(\cdot),T^*_{\tilde{p}})$ satisfying the hard terminal constraint $x(T)=x_f$ to the stated feasibility tolerance, and we record the corresponding optimal objective value $J^*_{\tilde{p}}$.

Figure~\ref{fig:sensitivity} reports the relative change in optimal cost, $(J^*_{\tilde{p}}-J^*_0)/J^*_0$, where $J^*_0$ denotes the nominal optimum. In the tested regime, the optimal cost is comparatively less sensitive to moderate variations in immune-related parameters (e.g., immune stimulation $a$ and intrinsic immunogenicity $\sigma$) and more sensitive to parameters governing tumor growth kinetics. Overall, the re-optimization-based sensitivity results indicate that the qualitative dosing strategy and its performance are not narrowly tuned to a single parameter realization, and that moderate mismatch does not qualitatively destabilize the derived schedules.

\subsection{Comparative synthesis}
\label{subsec:comparison_synthesis}

Taken together, the three case studies show that the qualitative structure of the optimal therapy is governed by exposure regularization acting through saturating pharmacodynamics. Case~A (balanced regularization) admits mixed boundary--interior schedules; Case~B (asymmetric regularization) yields channel-specific structure, with interior modulation appearing only in the regularized channel and boundary selection in the unregularized channel; and Case~C (minimum-time limit) recovers the bang--bang-type boundary-selected limit. Quantitative outcomes are summarized in Table~\ref{tab:comparison}. Relative to the minimum-time baseline (Case~C), Case~A reduces cumulative chemotherapy exposure while maintaining a comparable terminal time. Case~B further reduces chemotherapy exposure by shifting more of the therapeutic burden toward immunotherapy, thereby making the trade-off between cytotoxic load, treatment speed, and reliance on immune-mediated control explicit.

% ---------------------------------------------------------------
%  TABLE 3 — Comparative performance metrics
%  FIX: replaced {lccc} raw tabular with tabularx + \small;
%       used C columns so long "Structure" entries wrap cleanly.
% ---------------------------------------------------------------
\begin{table}[!htbp]
\centering
\small
\caption{Comparative performance metrics. $D_u=\int_0^T u(t)\,dt$ denotes
  cumulative chemotherapy exposure and $D_v=\int_0^T v(t)\,dt$ denotes
  cumulative immunotherapy exposure (normalized dose-time units).
  Percentage reductions are computed relative to the minimum-time
  baseline (Case~C) using $100\,(D_u^{\text{case}}-D_u^{C})/D_u^{C}$.}
\label{tab:comparison}
\begin{tabularx}{\linewidth}{@{}lCCC@{}}
\toprule
\textbf{Metric}
  & \makecell{\textbf{Case A}\\\textbf{(Balanced)}}
  & \makecell{\textbf{Case B}\\\textbf{(Asymmetric)}}
  & \makecell{\textbf{Case C}\\\textbf{(Min-Time)}} \\
\midrule
Time ($T$)     & 4.917  & 4.723  & 4.417 \\
Chemo ($D_u$)  & 2.771  & 2.528  & 3.401 \\
Immuno ($D_v$) & 1.785  & 4.723  & 4.417 \\
Structure      & Boundary + interior & Channel-specific & Bang--bang-type \\
\midrule
Chemo reduction vs.\ Case C & $-18.5\%$ & $-25.7\%$ & Baseline \\
\bottomrule
\end{tabularx}
\end{table}

\FloatBarrier

\section{Conclusion}
\label{sec:conclusion}

This paper studied a threshold-reaching chemo--immunotherapy scheduling problem in which the therapy channels follow saturating (Michaelis--Menten) pharmacodynamics. The central structural result is that, with strictly positive exposure penalties, the Hamiltonian minimization in each control channel admits an explicit three-regime law: depending on the switching potential, the optimal input is chosen at the lower bound, at the upper bound, or as a unique interior minimizer. Whenever an interior regime is active, the strict Legendre condition holds, so the resulting titration segments are regular PMP interior arcs rather than singular-control phenomena. This yields a clear structural contrast with control-affine formulations, where linear-in-control dynamics coupled with linear exposure costs drive bang--bang-type minimizers generically (absent singular arcs), whereas saturating pharmacodynamics make interior operation optimal without requiring quadratic smoothing terms.

Numerically, the saturation-induced rational nonlinearities produce a nonconvex direct transcription. To obtain reliable solutions, we used a hybrid QPSO--SQP strategy: QPSO provides global exploration to identify a high-quality basin of attraction, and SQP enforces feasibility and first-order optimality on a refined collocation grid. Across Cases~A--C and the additional diagnostic studies, the computed solutions satisfy the hard terminal tumor constraint and exhibit arc structures consistent with the PMP characterization, including the predicted transitions between boundary and interior regimes.

Several limitations should be noted. The model is deliberately minimal: it does not include explicit pharmacokinetics, treatment delays, state-path toxicity constraints, or tumor heterogeneity/resistance mechanisms, and the therapy coefficients are illustrative under normalized dosing. These extensions are important for patient-specific translation and will be pursued in future work. Nevertheless, within the present scope, the results establish a principled link between saturating pharmacodynamics and optimal-control structure: exposure regularization coupled with saturation is sufficient to generate interpretable, continuously modulated dosing arcs as direct consequences of PMP.

\section*{Acknowledgments}
The authors acknowledge support from the National Science Foundation under CAREER Award No.~2238269 (FAIN: 2238269) awarded to Dr.\ Shuo Wang.

%\clearpage
\appendix

\section{\texorpdfstring{Detailed derivation of the adjoint system}{Detailed Derivation of the Adjoint System}}
\label{app:adjoint}

The necessary conditions for optimality require the existence of a costate vector
$\bm{\lambda}(t) = [\lambda_1(t), \lambda_2(t), \lambda_3(t)]^T$
satisfying the adjoint differential equations
\[
\dot{\bm{\lambda}}(t) = -\nabla_{\bm{w}} H(\bm{w}(t),u(t),v(t),\bm{\lambda}(t)),
\qquad \bm{w}(t)=[x(t),y(t),z(t)]^T.
\]

The resulting expressions match the summary reported in Section~\ref{sec:necessary_optimality}.

The Hamiltonian is
\begin{align*}
    H &= \ell_1 x + \ell_2 y + \ell_3 z + A u + B v + C \\
      &\quad+ \lambda_1 \left[ \xi x \left(1 - \frac{x}{x_\infty}\right) - \theta x y - \frac{\alpha x u}{1+u} \right] \\
      &\quad+ \lambda_2 \left[ a(1 - bx)yz + \gamma - \delta y - \frac{\kappa y u}{1+u} + \frac{\nu y v}{1+v} \right] \\
      &\quad+ \lambda_3 \left[ \sigma x + \frac{\psi x u}{1+u} - \mu z \right].
\end{align*}

\subsection{\texorpdfstring{Adjoint equation for tumor volume ($\lambda_1$)}{Adjoint equation for tumor volume (lambda1)}}
Differentiating $H$ with respect to $x$ gives
\[
\frac{\partial H}{\partial x}
= \ell_1
+ \lambda_1\!\left(\xi\left(1-\frac{2x}{x_\infty}\right) - \theta y - \frac{\alpha u}{1+u}\right)
- \lambda_2 a b y z
+ \lambda_3\!\left(\sigma + \frac{\psi u}{1+u}\right).
\]

Hence,
\[
\dot{\lambda}_1
= -\ell_1
- \lambda_1\!\left(\xi\left(1-\frac{2x}{x_\infty}\right) - \theta y - \frac{\alpha u}{1+u}\right)
+ \lambda_2 a b y z
- \lambda_3\!\left(\sigma + \frac{\psi u}{1+u}\right).
\]

\subsection{\texorpdfstring{Adjoint equation for immune density ($\lambda_2$)}{Adjoint equation for immune density (lambda2)}}
Differentiating $H$ with respect to $y$ yields
\[
\frac{\partial H}{\partial y}
= \ell_2
- \lambda_1 \theta x
+ \lambda_2\!\left(a(1-bx)z - \delta - \frac{\kappa u}{1+u} + \frac{\nu v}{1+v}\right).
\]

Thus,
\[
\dot{\lambda}_2
= -\ell_2
+ \lambda_1 \theta x
- \lambda_2\!\left(a(1-bx)z - \delta - \frac{\kappa u}{1+u} + \frac{\nu v}{1+v}\right).
\]

\subsection{\texorpdfstring{Adjoint equation for antigen concentration ($\lambda_3$)}{Adjoint equation for antigen concentration (lambda3)}}
Differentiating $H$ with respect to $z$ gives
\[
\frac{\partial H}{\partial z}
= \ell_3
+ \lambda_2 a(1-bx)y
- \lambda_3 \mu.
\]

Hence,
\[
\dot{\lambda}_3
= -\ell_3
- \lambda_2 a(1-bx)y
+ \lambda_3 \mu.
\]

\section{\texorpdfstring{Analytic resolution of interior control arcs}{Analytic resolution of interior control arcs}}
\label{app:arcs}

The pointwise minimizers are obtained from the stationarity conditions
$\partial H/\partial u = 0$ and $\partial H/\partial v = 0$ on intervals where
the optimal controls lie strictly inside their bounds, i.e.,
$0<u<u_{\max}$ and $0<v<v_{\max}$. Outside these interior regimes, the minimizers
are attained at the admissible bounds as characterized in
Proposition~\ref{prop:pointwise_minimizers}.

\subsection{\texorpdfstring{Chemotherapy control $u(t)$}{Chemotherapy control u(t)}}
Define the switching potential
\[
\Phi_u(t) = -\lambda_1(t)\,\alpha x(t) \;-\; \lambda_2(t)\,\kappa y(t) \;+\; \lambda_3(t)\,\psi x(t).
\]

Since $\frac{d}{du}\!\left(\frac{u}{1+u}\right)=\frac{1}{(1+u)^2}$, the stationarity condition yields
\[
\frac{\partial H}{\partial u}
= A + \frac{\Phi_u(t)}{(1+u)^2} = 0
\quad\Longrightarrow\quad
(1+u)^2 = -\frac{\Phi_u(t)}{A}.
\]

Assuming $A>0$, a real interior stationary point requires $\Phi_u(t)<0$, and
\[
\hat{u}(t) = -1 + \sqrt{\frac{-\Phi_u(t)}{A}}.
\]

Projecting onto the admissible set $u\in[0,u_{\max}]$ gives
\begin{equation}
u^*(t) = \min\!\left\{u_{\max},\; \max\!\left\{0,\; -1 + \sqrt{\frac{-\Phi_u(t)}{A}}\right\}\right\}.
\end{equation}

\subsection{\texorpdfstring{Immunotherapy control $v(t)$}{Immunotherapy control v(t)}}
Define the switching potential
\[
\Phi_v(t) = \lambda_2(t)\,\nu y(t),
\]
so that stationarity gives
\[
\frac{\partial H}{\partial v}
= B + \frac{\Phi_v(t)}{(1+v)^2} = 0
\quad\Longrightarrow\quad
(1+v)^2 = -\frac{\Phi_v(t)}{B}.
\]

Assuming $B>0$, a real interior stationary point requires $\Phi_v(t)<0$, and
\[
\hat{v}(t) = -1 + \sqrt{\frac{-\Phi_v(t)}{B}}.
\]

Projecting onto the admissible set $v\in[0,v_{\max}]$ yields
\begin{equation}
v^*(t) = \min\!\left\{v_{\max},\; \max\!\left\{0,\; -1 + \sqrt{\frac{-\Phi_v(t)}{B}}\right\}\right\}.
\end{equation}

\noindent \textit{Numerical implementation note.}
In finite-precision arithmetic, $\Phi_u$ and $\Phi_v$ may approach zero from below
(e.g., $\Phi_u \approx -10^{-16}$), which can cause spurious complex values inside the square root.
We therefore evaluate
\[
\sqrt{\max\!\left\{0,\; -\frac{\Phi_u(t)}{A}\right\}}, \qquad
\sqrt{\max\!\left\{0,\; -\frac{\Phi_v(t)}{B}\right\}},
\]
before applying the projection onto $[0,u_{\max}]$ and $[0,v_{\max}]$.

\medskip
Received for publication January 16, 2026; early access May 6, 2026
\medskip
\end{document}